\documentclass[12pt,leqno]{amsart}
\usepackage[a4paper,top=2cm,bottom=2cm,left=2cm,right=2cm]{geometry}

\usepackage{tikz}
\usetikzlibrary{positioning, matrix, fit}

\usetikzlibrary{arrows}

\usepackage{natbib}
\usepackage{chngcntr}
\usepackage{lineno,hyperref}
\usepackage{rotating}
\usepackage{amsfonts}
\usepackage{amsmath}
\usepackage{amsthm}
\usepackage{xcolor}
\usepackage{amssymb,amsmath}
\usepackage{subfig}
\usepackage[english]{babel}
\usepackage{graphicx}
\usepackage{tabularx}
\usepackage{enumerate, hyperref}

\theoremstyle{plain}
\newtheorem{theorem}{Theorem}[section]

\newtheorem{lemma}[theorem]{Lemma}
\newtheorem{proposition}[theorem]{Proposition}

\theoremstyle{definition}
\newtheorem{definition}[theorem]{Definition}
\newtheorem{example}[theorem]{Example}

\theoremstyle{remark}

\title{New formulas for $\pi$ involving infinite nested square roots and Gray code.}

\author[Pierluigi Vellucci, Alberto Maria Bersani]{Pierluigi Vellucci $^*$
   \and Alberto Maria Bersani$^{**}$}

\begin{document}
\maketitle

\begin{abstract}
In previous papers we introduced a class of polynomials which follow the same recursive formula as the Lucas-Lehmer numbers, studying the distribution of their zeros and remarking that this distribution follows a sequence related to the binary Gray code. It allowed us to give an order for all the zeros of every polynomial $L_n$. In this paper, the zeros, expressed in terms of nested radicals, are used to obtain two formulas for $\pi$: the first can be seen as a generalization of the known formula
$$\pi=\lim_{n\rightarrow \infty} 2^{n+1}\cdot \sqrt{2-\underbrace{\sqrt{2+\sqrt{2+\sqrt{2+...+\sqrt{2}}}}}_{n}} \ ,$$
related to the smallest positive zero of $L_n$; the second is an exact formula for $\pi$ achieved thanks to some identities valid for $L_n$.
\end{abstract}

\section{Introduction.}

In this{\let\thefootnote\relax\footnotetext{Keywords: $\pi$ formulas; Gray code; Continued roots; Nested square roots; zeros of Chebyshev polynomials. 2010 MSC: 40A05, 11Y60.}} paper{\let\thefootnote\relax\footnotetext{$^*$Department of Economics, University of Roma TRE, via Silvio D'Amico 77, 00145 Rome, Italy (pierluigi.vellucci@uniroma3.it). $^{**}$ Department of Mechanical and Aerospace Engineering, Sapienza University, Via Eudossiana n. 18, 00184 Rome, Italy (alberto.bersani@sbai.uniroma1.it).}} we obtain $\pi$ as the limit of a sequence related to the zeros of the class of polynomials $L_n(x) = L_{n-1}^2(x) - 2$ created by means of the same iterative formula used to build the well-known Lucas-Lehmer sequence, employed in primality tests \cite{Bres89,Fin03,Kos01,Leh30,Luc878a,Luc878b,Luc878c,Rib88}. This class of polynomials was introduced in previous papers \cite{Vel}, \cite{VelBer}.

The results obtained here are based on the placement of the zeros of the polynomials $L_n(x)$, studied in \cite{Vel2}, and generalize the well-known relation
\begin{equation}
\label{eq:Munkhammar}
\pi=\lim_{n\rightarrow \infty} 2^{n+1}\cdot \sqrt{2-\underbrace{\sqrt{2+\sqrt{2+\sqrt{2+...+\sqrt{2}}}}}_{n}} \ .
\end{equation}
Zeros have a structure of nested radicals, by means of which we can build infinite sequences of nested radicals converging to $\pi$. The ordering of the zeros follows the sequence of the binary Gray code, which is very useful in computer science and in telecommunications \cite{Vel2}.

Starting from the seminal papers by Ramanujan (\cite{Rama}, \cite{Ber} pp. 108-112), there is a vast literature studying the properties of the so-called continued radicals as, for example: \cite{Her,BordeB,Siz,Joh,Eft2,Ly}. Other authors investigated the properties of more general continued operations and their convergence. For a nice review of these results, see for example \cite{Jones}, which focuses mainly on continued reciprocal roots.

The nested square roots of 2 are a special case of the wide class of continued radicals. They have been studied by several authors. In particular, Cipolla \cite{Cipolla} obtained a very elegant formula for
$$\sqrt{2+i_{n}\sqrt{2+i_{n-1}\sqrt{2+\cdots+i_1\sqrt{2}}}}\, ,$$
in terms of $2 \cos\left( k_n \frac{\pi}{2^{n+2}}\right)$, where $i_k\in\{-1,1\}$ and $k_n$ is a constant depending on $i_1$, $\dots$, $i_n$. 

Servi \cite{Ser03} rediscovered and extended Cipolla's formula, tying the evaluation of nested square roots of the form
\begin{equation}
\label{eq:servi}
R(b_k,...,b_1)=\frac{b_k}{2}\sqrt{2+b_{k-1}\sqrt{2+b_{k-2}\sqrt{2+...+b_2\sqrt{2+2\sin\left(\frac{b_1 \pi}{4}\right)}}}}
\end{equation}
where $b_i\in\{-1,0,1\}$ for $i\neq 1$, to expression
\begin{equation}
\label{eq:servi2}
\left(\frac{1}{2}-\frac{b_k}{4}-\frac{b_k b_{k-1}}{8}-...-\frac{b_k b_{k-1} ... b_1}{2^{k+1}}\right)\pi
\end{equation}
to obtain, amongst other results, some nested square roots representations of $\pi$:
\begin{equation}
\label{eq:servi3}
\pi=\lim_{k\rightarrow\infty} \left[\frac{2^{k+1}}{2-b_1}R\left(\underbrace{1,-1,1,1,\dots,1,1,b_1}_{k\ \text{terms}}\right)\right]
\end{equation}
where $b_1\neq 2$. Nyblom \cite{Nyb}, citing Servi's work, derived a closed-form expression for (\ref{eq:servi}) with a generic $x\geq2$ that replaces $2\sin\left(\frac{b_1 \pi}{4}\right)$ in (\ref{eq:servi}).
Efthimiou \cite{Eft} proved that the radicals
$$a_0 \sqrt{2+a_1\sqrt{2+a_2\sqrt{2+a_3\sqrt{2+\cdots}}}}, \ \ a_i\in\{-1,1\}$$
have limits two times the fixed points of the Chebyshev polynomials $T_{2^n}(x)$, unveiling an interesting relation between these topics. The previous formula is equivalent to \begin{equation}
\label{eq:prop2a}
\pm  \sqrt{2\pm\sqrt{2\pm\sqrt{2\pm\sqrt{2\pm...\pm\sqrt{2}}}}}\, .
\end{equation}
In \cite{More,More2}, the authors report a relation between the nested square roots of depth $n$ as $\pm  \sqrt{2\pm\sqrt{2\pm\sqrt{2\pm\cdots\pm\sqrt{2+2 z}}}}$, $z\in\mathbb C$ and the Chebyshev polynomials of degree $2^n$ in a complex variable, generalizing and unifying Servi and Nyblom's formulas. In \cite{More}, the authors propose an ordering of the continued roots
\begin{equation}\label{eq:morenoradice}
b_k\sqrt{2+b_{k-1}\sqrt{2+\cdots+b_1\sqrt{2+2\xi}}}\ ,
\end{equation}
where $\xi=1$ and each $b_i$ is either $1$ or $-1$, according to formula
\begin{equation}\label{eq:moreno}
j(b_k,\dots,b_1)=\frac{1}{2}\left(2^k-\left(\sum_{j=1}^k\left(2^{k-j}\prod_{i=0}^{j-1} b_{k-i}\right)\right) + 1\right),
\end{equation}
for each positive integer $k$. Formula (\ref{eq:morenoradice}) expresses the nested square roots of 2 in (\ref{eq:prop2a}), and in \cite{Vel2} we gave an alternative ordering for them involving the so-called Gray code which, to the best of our knowledge, is applied for the first time to these topics.

Actually, there is a strong connection between \cite{More} and \cite{Vel2}. From (\ref{eq:moreno}), we have, for example: $j(1,1,1) =1$, $j(1,1,-1) =2$, $j(1,-1,-1) =3$, $j(1,-1,1)  =4$, $j(-1,-1,1)  =5$, j(-1,-1,-1)  =6, $j(-1,1,-1)  =7$ and $j(-1,1,1)  =8$. If we associate bit $0$ to number $b_i=1$ and bit $1$ to number $b_i=-1$, in the expression of index $j$, we obtain
\begin{align*}
(1,1,1) & \ \ \mapsto \ \ (0,0,0) \\
(1,1,-1)& \ \ \mapsto \ \ (0,0,1) \\
(1,-1,-1)& \ \ \mapsto \ \ (0,1,1) \\
(1,-1,1)& \ \ \mapsto \ \ (0,1,0) \\
(-1,-1,1)& \ \ \mapsto \ \ (1,1,0) \\
(-1,-1,-1)& \ \ \mapsto \ \ (1,1,1) \\
(-1,1,-1)&  \ \ \mapsto \ \ (1,0,1) \\
(-1,1,1)& \ \ \mapsto \ \ (1,0,0) \ ,
\end{align*}
which are just an example of Gray code.


In this paper, after having recalled in Section \ref{sec:prelim} the most important definitions and properties of the Lucas-Lehmer polynomials, in Section \ref{sec:result} we extend formula (\ref{eq:servi3}), using the properties of the zeros of these polynomials (shown in \cite{Vel2}), stating and proving Theorem \ref{th:gray}, which produces infinite numerical sequences converging to $\pi$. Besides, under suitable assumptions, Proposition \ref{prop:terza} simplifies the expression of (\ref{eq:servi2}) listed in Servi's Theorem (\cite{Ser03}, formula (8)).

We also show that the generalizations of the Lucas-Lehmer map, $M_n^a$ for $a>0$ introduced in \cite{Vel}, have the same properties of $L_n$, for what concerns the distribution of the zeros and the approximations of $\pi$. We also obtain $\pi$ not as the limit of a sequence, but equal to an expression involving the zeros of the polynomials $L_n$ and $M_n^a$ for $a>0$. 

Some perspectives of future applications of our results are reported in Section \ref{sec:conc}.

\section{Preliminaries.}
\label{sec:prelim}

In this section we recall properties and useful results from our previous papers (\cite{VelBer}, \cite{Vel}, \cite{Vel2}), and therefore we will list them without proofs. 

\subsection{The class of Lucas-Lehmer polynomials}

We recall below some basic facts about Lucas-Lehmer polynomials
\begin{equation}
\label{eq:Lucas_Lehmer}
L_0(x) = x \quad ; \quad L_n(x) = L_{n-1}^2 - 2 \ \ \forall n \geq 1
\end{equation}
taken from \cite{Vel}. The polynomials $L_n(x)$ are orthogonal with respect to the weight function
$$\frac{1}{4\sqrt{4-x^2}}$$
defined on $x\in(-2,2)$. 

Besides, for each $n\geq1$ we have
\begin{equation}
\label{eq:propcheby1}
L_{n}(x)=2\ T_{2^{n-1}}\left(\frac{x^{2}}{2}-1\right)
\end{equation}
where the \emph{Tchebycheff polynomials of first kind} \cite{Bate53,Riv90} satisfy the recurrence relation
\begin{displaymath}
\begin{cases}
T_{n}(x)=2xT_{n-1}(x)-T_{n-2}(x) \qquad n \geq 2 \\
T_{0}(x)=1, \ \ T_{1}(x)=x \\
\end{cases}
\end{displaymath}
from which it easily follows that for the $n$-th term:
\begin{equation}
T_{n}(x)=\frac{\left(x -\sqrt{x^2-1} \right)^{n}+\left( x +\sqrt{x^2-1} \right)^{n}}{2}\, .
\end{equation}
This formula is valid in $\mathbb R$ for $|x| \geq 1$; here we assume instead that $T_n$, defined in $\mathbb R$, can take complex values, too. Let $x=2\cos\theta$, then the polynomials $L_{n}(x)$ admit the representation
\begin{equation}
\label{formulaconcos}
L_{n}(2\cos\theta)=2\cos\left(2^{n} \theta\right)\, .
\end{equation}

When $|x| \leq 2$, we can write $x = 2 \cos(\vartheta)$, thus $\displaystyle \frac{x^2}{2} - 1 = \cos(2 \vartheta)$; hence, for $|x| \neq \sqrt{2}$, we can also put
\begin{equation}
\label{theta}
\vartheta(x)= \frac{1}{2}\arctan\left[\frac{\sqrt{1-\left(\frac{x^{2}}{2}-1\right)^2}}{\frac{x^{2}}{2}-1} \right]+b\pi
\end{equation}
where $b$ is a binary digit; thus, using (\ref{formulaconcos}), we obtain
\begin{equation}
L_{n}(x)=2\cos\left(2^{n}\vartheta(x)\right) \ .
\end{equation}
Moreover, since $L_1(\pm \sqrt{2}) = 0 \ ; \ L_2(\pm \sqrt{2}) = - 2 \ ; \  L_n(\pm \sqrt{2}) = 2 \quad \forall n \ge 3 \ $, then the argument of $L_n(\pm \sqrt{2})$ is $0$ for every $n \ge 3$.

By setting further
\begin{equation}
\label{theta_value}
\theta(x)= \frac{1}{2}\arctan\left[\frac{\sqrt{1-\left(\frac{x^{2}}{2}-1\right)^2}}{\frac{x^{2}}{2}-1} \right]
\end{equation}
we can write:
\begin{equation}
\label{Ln_value}
L_{n}(x)=2\cos\left(2^{n}\theta(x)+2^{n}b\pi\right)=2\cos\left(2^{n}\theta(x)\right) \ .
\end{equation}

Like those of the first kind, \emph{Tchebycheff polynomial of second kind} are defined by a recurrence relation ~\cite{Bate53,Riv90}:
\begin{displaymath}
\begin{cases}
U_{0}(x)=1, \ \ U_{1}(x)=2x \\
U_{n}(x)=2xU_{n-1}(x)-U_{n-2}(x) \quad \forall n \geq 2\, ,\\
\end{cases}
\end{displaymath}
which is satisfied by
\begin{equation}
\label{eq:Usemplice}
U_{n}(x)=\sum_{k=0}^n (x + \sqrt{x^2 -1})^k (x - \sqrt{x^2-1})^{n-k} \quad \forall x \in [-1, 1] \ .
\end{equation}

This relation is equivalent to
\begin{equation}
\label{eq:Ufrazione}
U_{n}(x)=\frac{\left(x +\sqrt{x^2-1} \right)^{n+1}-\left( x -\sqrt{x^2-1} \right)^{n+1}}{2\sqrt{x^2-1}}
\end{equation}
for each $x \in (-1, 1)$. From continuity of function (\ref{eq:Usemplice}), we observe that (\ref{eq:Ufrazione}) can be extended by continuity in $x = \pm 1$, too.

It can therefore be put $U_n(\pm 1) = (\pm 1)^n (n+1)$ in (\ref{eq:Ufrazione}). From \cite{Vel}, for each $n\geq1$ we have
\begin{equation}
\label{eq:propcheby2}
\prod_{i=1}^{n}L_{i}(x)=U_{2^{n}-1}\left(\frac{x^{2}}{2}-1\right)
\end{equation}
and
\begin{equation}
\label{euler}
\prod_{i=1}^{n}L_{i}(2\cos\theta)= \frac{\sin\left(2^{n+1} \theta \right)}{\sin2\theta} \ .
\end{equation}

\subsection{An ordering for zeros of Lucas-Lehmer polynomials using Gray code.}
Given a binary code, its \emph{order} is the number of bits with which the code is built, while its \emph{length} is the number of strings that compose it. The celebrated Gray code \cite{Gard86,Nij78} is a binary code of order $n$ and length $2^{n}$.
\begin{figure}[h!]
  \begin{center}
\begin{tikzpicture}
\matrix (A) [matrix of nodes]
{0 & 0 & 0 & 0 \\
0 & 0 & 0 & 1 \\
0 & 0 & 1 & 1 \\
0 & 0 & 1 & 0 \\
0 & 1 & 1 & 0 \\
0 & 1 & 1 & 1 \\
0 & 1 & 0 & 1 \\
0 & 1 & 0 & 0 \\
1 & 1 & 0 & 0 \\
1 & 1 & 0 & 1 \\
1 & 1 & 1 & 1 \\
1 & 1 & 1 & 0 \\
1 & 0 & 1 & 0 \\
1 & 0 & 1 & 1 \\
1 & 0 & 0 & 1 \\
1 & 0 & 0 & 0 \\
};

\node[draw=blue, thick, inner sep=2pt, fit=(A-9-2.north west) (A-16-4.south east)] (BB) {};
\node[draw=red, thick, inner sep=2pt, fit=(A-13-3.north west) (BB.south east)] (BA) {};

\node[red, right= 3mm of BA.south east] (LA) {encapsulated sub-code};
\node[blue] at (BB-|LA) {encapsulated sub-code};

\end{tikzpicture}
   \end{center}
\caption{Sub-codes for $m=2$, $m=3$.}
\label{fig:gray}
 \end{figure}
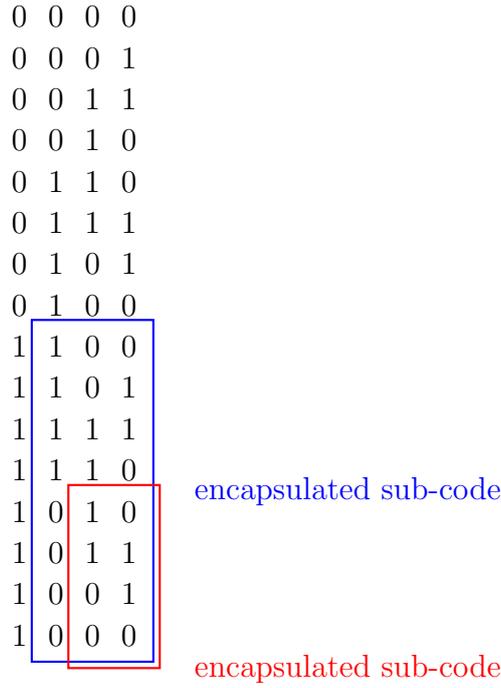
We recall below how a Gray Code is generated; if the code for $n-1$ bit is formed by binary strings
$$\begin{array}{l}
g_{n-1,1} \\
\vdots\\
g_{n-1,2^{n-1}-1}\\
g_{n-1,2^{n-1}}\ ,
  \end{array}
$$
the code for $n$ bits is built from the previous one in the following way:
$$\begin{array}{l}
0g_{n-1,1} \\
\vdots \\
0g_{n-1,2^{n-1}-1} \\
0g_{n-1,2^{n-1}} \\
1g_{n-1,2^{n-1}} \\
1g_{n-1,2^{n-1}-1}\\
\vdots\\
1g_{n-1,1}\ .
\end{array}$$

Just as an example, we have, for $n=1$: $g_{1,1}=0 \ ; \ g_{1,2}=1$; for $n=2$: $g_{2,1}=00 \ ; \ g_{2,2}=01 \ ; \ g_{2,3}=11 \ ; \ g_{2,4}=10$; for $n=3$:
\begin{align}\label{eq:gray3}
 g_{3,1}&=000   &   g_{3,2}&=001   &   g_{3,3}&=011   &   g_{3,4}&=010 \notag \\
 g_{3,5}&=110   &   g_{3,6}&=111   &   g_{3,7}&=101   &   g_{3,8}&=100 \ ,
\end{align}
and so on.

Following the notation introduced in \cite{Vel2}, we recall some preliminaries about Gray code.
\begin{definition}
\label{def:prima}
Let us consider a Gray code of order $n$ and length $2^{n}$. A \emph{sub-code} is a Gray code of order $m<n$ and length $2^{m}$.
\end{definition}
\begin{definition}
\label{def:seconda}
Let us consider a Gray code of order $n$ and length $2^{n}$. An \emph{encapsulated sub-code} is a sub-code built starting from the last string of Gray code of order $n$ that contains it.
\end{definition}
Figure (\ref{fig:gray}) contains some examples of encapsulated sub-codes inside a Gray code (with order $4$ and length $16$).

Let us consider the signs ``plus'' and ``minus'' in the nested form that expresses generic zeros of $L_{n}$,  as follows:
\begin{equation}
\label{eq:annidata}
\sqrt{2\pm\underbrace{\sqrt{2\pm\sqrt{2\pm\sqrt{2\pm...\pm\sqrt{2\pm\sqrt{2}}}}}}}\, .
\end{equation}
The underbrace encloses $n-1$ signs ``plus'' or ``minus'', each one placed before each nested radical. Starting from the first nested radical we apply a code (i.e., a system of rules) that associates bits $0$ and $1$ to ``plus'' and ``minus'' signs, respectively.

Obviously, it is possible to obtain $2^{n-1}$ strings formed by $n-1$ bits. Let us define with $\displaystyle \{\omega(g_{n-1, j}) \}_{j=1, ..., 2^{n-1}}$ the set of all the $2^{n-1}$ nested radicals of the form
\begin{equation*}
2\pm\underbrace{\sqrt{2\pm\sqrt{2\pm\sqrt{2\pm...\pm\sqrt{2\pm\sqrt{2}}}}}}_{n-1 \ \ signs}\, , 
\end{equation*}
where each element of the set differs from the others for the sequence of ``plus'' and ``minus'' signs, and the index $j$ determines the position of the nested radical by virtue of Gray code.

\section{Main results: $\pi$ formulas involving nested radicals.}
\label{sec:result}

\subsection{Infinite sequences tending to $\pi$}

Let us consider two finite sequences $x=\{x_1,\dots,x_n\}$ and $y=\{y_1, \dots, y_m \}$, $n,m\in\mathbb N$. We define the \emph{concatenation} of these sequences the sequence $xy:=\{x_1,\dots,x_n,$ $y_1,\dots,y_n\}$. Let us also consider the binary string $b_{n-m}$ composed by $n-m$ bits. The following results concern the set $\{\omega\left( b_{n-m}\, g_{m,h} \right)\}_{h=0}^{2^m-1}$. For example, in the following Lemma \ref{l:1} we have $b_{n-m}=b_2=10$.

\begin{figure}[tb]
\centering
\includegraphics[scale=0.80]{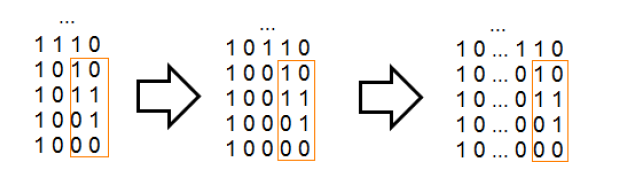}
\caption{A possible subcode (orange), where the meaning of the limit (\ref{eq:theorem_gray}) is highlighted: in this way the number of symbols 0, on the left of the sub-code, increases.}
\label{fig:limite}
\end{figure}

\begin{lemma}
\label{l:1}
For all $m\in\mathbb N$ one has:
\begin{equation}
\label{eq:aste}
\sqrt{\omega(10g_{m,h+1})}=2\sin\left(\frac{2h+1}{2^{m+4}}\pi\right) \ \ \ h\in\{0,1,\dots,2^m-1\}
\end{equation}
\end{lemma}
\begin{proof}
We proceed with induction principle for $m$ to prove (\ref{eq:aste}). If $m=1$:
\begin{equation}
\sqrt{\omega(10g_{1,h+1})}=2\sin\left(\frac{2h+1}{2^{5}}\pi\right) \ \ \ h\in\{0,1\}
\end{equation}
i.e.
$$\sqrt{\omega(10g_{1,1})}=2\sin\left(\frac{\pi}{2^{5}}\right)$$
for $h=0$, and
$$\sqrt{\omega(10g_{1,2})}=2\sin\left(\frac{3\pi}{2^{5}}\right)$$
for $h=1$, where $g_{1,1}=0$ and $g_{1,2}=1$. These formulas are easy to check. Now we are going to check (\ref{eq:aste}) for $m+1$:
\begin{equation}
\label{eq:astepiuno}
\sqrt{\omega(10g_{m+1,h+1})}=2\sin\left(\frac{2h+1}{2^{m+5}}\pi\right) \ \ \ h\in\{0,1,\dots,2^{m+1}-1\}
\end{equation}
having assumed it true for $m\geq 1$.
From Gray Code's definition we have that either a) $g_{m+1,h+1}=0g_{m,h+1}$ or b) $g_{m+1,h+1}=1g_{m,2^m-h}$. In the former case:
$$\sqrt{\omega(10g_{m+1,h+1})}=\sqrt{\omega(100g_{m,h+1})}$$
where
\begin{align}
\label{eq:aste2}
\sqrt{\omega(100g_{m,h+1})}&=\sqrt{2-\sqrt{\omega(00g_{m,h+1})}} \notag \\
&=\sqrt{2-\sqrt{2+\sqrt{\omega(0g_{m,h+1})}}}\, .
\end{align}
But in fact: $\omega(10g_{m,h+1})=2-\sqrt{\omega(0g_{m,h+1})}$, so (\ref{eq:aste2}) becomes
\begin{align}
\sqrt{\omega(100g_{m,h+1})}&=\sqrt{2-\sqrt{2+\sqrt{\omega(0g_{m,h+1})}}} \notag \\
&=\sqrt{2-\sqrt{4-\omega(10g_{m,h+1})}} \notag \\
&=\sqrt{2-\sqrt{4-4\sin^2\left(\frac{2h+1}{2^{m+4}}\pi\right)}} \notag \\
&=\sqrt{2-2\cos\left(\frac{2h+1}{2^{m+4}}\pi\right)} \notag \\
&=2\sin\left(\frac{2h+1}{2^{m+5}}\pi\right)\, .
\end{align}
Therefore (\ref{eq:astepiuno}) is proved for the case a). 

Now we assume that $g_{m+1,h+1}=1g_{m,2^m-h}$:
$$\sqrt{\omega(10g_{m+1,h+1})}=\sqrt{\omega(101g_{m,2^m-h})}$$
thus
\begin{align}
\sqrt{\omega(101g_{m,2^m-h})}&=\sqrt{2-\sqrt{\omega(01g_{m,2^m-h})}} \notag \\
&=\sqrt{2-\sqrt{2+\sqrt{\omega(1g_{m,2^m-h})}}} \notag \\
&=\sqrt{2-\sqrt{2+\sqrt{2-\sqrt{\omega(g_{m,2^m-h})}}}}
\end{align}
Noting that
$$\omega(0g_{m,2^m-h})=2+\sqrt{\omega(g_{m,2^m-h})}$$
it follows that
\begin{equation}
\label{eq:aste3}
\sqrt{\omega(101g_{m,2^m-h})}=\sqrt{2-\sqrt{2+\sqrt{4-\omega(0g_{m,2^m-h})}}}
\end{equation}
From $\omega(10g_{m,2^m-h})=2-\sqrt{\omega(0g_{m,2^m-h})}$, equation (\ref{eq:aste3}) becomes
\begin{equation}
\label{eq:aste4}
\sqrt{\omega(101g_{m,2^m-h})}=\sqrt{2-\sqrt{2+\sqrt{4-\left[2-\omega(10g_{m,2^m-h})\right]^2}}}
\end{equation}
From (\ref{eq:aste}), we have
$$\sqrt{\omega(10g_{m,2^m-h})}=2\sin\left(\frac{2^{m+1}-(2h+1)}{2^{m+4}}\pi\right)$$
and equation (\ref{eq:aste4}) can be rewritten
\begin{align}
\sqrt{\omega(101g_{m,2^m-h})}&=\sqrt{2-\sqrt{2+\sqrt{4-\left[2-\omega(10g_{m,2^m-h})\right]^2}}}\notag \\
&=\sqrt{2-\sqrt{2+\sqrt{4-\left[2-4\sin^2\left(\frac{2^{m+1}-(2h+1)}{2^{m+4}}\pi\right)\right]^2}}}\notag \\
&=\sqrt{2-\sqrt{2+\sqrt{4-4\cos^2\left(\frac{2^{m+1}-(2h+1)}{2^{m+3}}\pi\right)}}}\notag \\
&=\sqrt{2-\sqrt{2+2\sin\left(\frac{2^{m+1}-(2h+1)}{2^{m+3}}\pi\right)}}\notag \\
&=\sqrt{2-\sqrt{2+2\cos\left(\frac{\pi}{2}-\frac{2^{m+1}-(2h+1)}{2^{m+3}}\pi\right)}}\notag \\
&=\sqrt{2-2\cos\left(\frac{\pi}{4}-\frac{2^{m+1}-(2h+1)}{2^{m+4}}\pi\right)}\, .
\end{align}
Accordingly:
$$\sqrt{\omega(101g_{m,2^m-h})}=2\sin\left(\frac{2(h+2^{m})+1}{2^{m+5}}\pi\right)\, .$$
Since the term $h+2^{m}\in\{2^m,\dots,2^{m+1}-1\}$ for $h\in\{0,\dots,2^m-1\}$, then (\ref{eq:astepiuno}) is fully shown and, with it, the whole proposition.
\end{proof}

\begin{proposition}
\label{prop:terza}
For each $n\geq m+2$, $h\in\mathbb N$ such that $h\in\{0,1,\dots,2^m-1\}$:
\begin{equation}
\label{eq:propterza}
\sqrt{\omega(\bar b_{n-m}g_{m,h+1})}=2\sin\left(\frac{2h+1}{2^{n+2}}\pi\right)\, ,
\end{equation}
where $\bar b_{n-m}=10\dots0$ has $n-m-1$ zeros.
\end{proposition}
\begin{proof}
Put $n-m=\delta_0$, $n-m-1=\delta_1$, $n-m-2=\delta_2$, $\dots$, $n-m-k=\delta_{k}$ for $0\leq k\leq n-m-2$. Let us proceed by means of induction principle on $n$. Fixing $m$, suppose formula (\ref{eq:propterza}) to be true for a generic index $\delta_1$,
\begin{equation}
\label{eq:n1}
\sqrt{\omega (\bar b_{\delta_1} g_{m,h+1})}=2\sin\left(\frac{2h+1}{2^{n+1}}\pi\right)\, , \quad \bar b_{\delta_1}=10\dots0\, ,
\end{equation}
where $\bar b_{\delta_1}$ has $\delta_1-1$ zeros --- and proceed to check the case $\delta_0$. We work on both sides of (\ref{eq:n1}):
\begin{align*}
\omega (\bar b_{\delta_1}g_{m,h+1})&=4\sin^2\left(\frac{2h+1}{2^{n+1}}\pi\right)  \\
2-\sqrt{\omega (\tilde b_{\delta_2}g_{m,h+1})}&=4-4\cos^2\left(\frac{2h+1}{2^{n+1}}\pi\right)\, ,  
\end{align*}
where $\tilde b_{\delta_2}=0\dots0$. Then
\begin{align}
&-\sqrt{\omega(\tilde b_{\delta_2}g_{m,h+1})}=2-4\cos^2\left(\frac{2h+1}{2^{n+1}}\pi\right) \notag \\
&2+\sqrt{\omega (\tilde b_{\delta_2} g_{m,h+1})}=4\cos^2\left(\frac{2h+1}{2^{n+1}}\pi\right) \notag \\
&\omega ( \tilde b_{\delta_1} g_{m,h+1})=4\cos^2\left(\frac{2h+1}{2^{n+1}}\pi\right)\, , \quad \tilde b_{\delta_1}=0\tilde b_{\delta_2}\, ,
\end{align}
whence
$$\sqrt{\omega (\tilde b_{\delta_1} g_{m,h+1})}=2\left|\cos\left(\frac{2h+1}{2^{n+1}}\pi\right)\right|=2\cos\left(\frac{2h+1}{2^{n+1}}\pi\right)\, .$$
Thus:
\begin{equation*}
  \sqrt{\omega (\tilde b_{\delta_1} g_{m,h+1})}=2\left(1-2\sin^2\left(\frac{2h+1}{2^{n+2}}\pi\right)\right)
\end{equation*}
or
\begin{equation*}
2-\sqrt{\omega (\tilde b_{\delta_1} g_{m,h+1})}=4\sin^2\left(\frac{2h+1}{2^{n+2}}\pi\right)\, ,
\end{equation*}
and
$$\omega(\bar b_{\delta_0} g_{m,h+1})=4\sin^2\left(\frac{2h+1}{2^{n+2}}\pi\right)\, ,$$
hence
$$\sqrt{\omega (\bar b_{\delta_0} g_{m,h+1})}=2\left|\sin\left(\frac{2h+1}{2^{n+2}}\pi\right)\right|=2\sin\left(\frac{2h+1}{2^{n+2}}\pi\right).$$
Let us remark that $\bar b_{\delta_0}=\bar b_{n-m}=10\dots0$ and it has $n-m-1$ zeros.

The absolute value can be removed by the proposition's assumptions. Therefore, the inductive step is proved. Let us consider the base step: $\delta_0=2$. Indeed:
$$\sqrt{\omega (10g_{m,h+1})}=2\sin\left(\frac{2h+1}{2^{n-m+2}2^m}\pi\right)$$
or,
\begin{equation}
\sqrt{\omega(10g_{m,h+1})}=2\sin\left(\frac{2h+1}{2^{m+4}}\pi\right) \ \ \ h\in\{0,1,\dots,2^m-1\}
\end{equation}
which is proved, for all $m\in\mathbb N$, in Lemma \ref{l:1}.
\end{proof}

\begin{theorem}
\label{th:gray}
\begin{equation}
\label{eq:theorem_gray}
\lim_{n\rightarrow\infty}\ \frac{2^{n+1}}{2h+1}\ \sqrt{\omega (\bar b_{n-m} g_{m,h+1})}=\pi
\end{equation}
where $\bar b_{n-m}=10\dots0$ has $n-m-1$ zeros, for every $h\in\mathbb N$ such that $h\in\{0,1,\dots,2^m-1\}$ and $n>m+1$.
\end{theorem}
\begin{proof}
From Proposition \ref{prop:terza}, we have
\begin{equation}
\frac{2^{n+1}}{2h+1}\ \sqrt{\omega (\bar b_{n-m} g_{m,h+1})}=\frac{2^{n+2}}{2h+1}\ \sin\left(\frac{2h+1}{2^{n+2}}\pi\right)
\end{equation}
that, for a well-known limit, tends to $\pi$ for $n \to \infty$.
\end{proof}
\begin{example}
With the help of computational tools we show below some iterations of a sequence described by
$$\frac{2^{n+1}}{2h+1}\ \sqrt{\omega (\bar b_{n-m} g_{m,h+1})}\, ,$$
where $\bar b_{n-m}=10\dots0$ has $n-m-1$ zeros.

Let us consider $m=3$; then
$$g_{3,1}=000 \ ; \ g_{3,2}=001 \ ; \ g_{3,3}=011 \ ; \ g_{3,4}=010 \ ;$$
$$g_{3,5}=110 \ ; \ g_{3,6}=111 \ ; \ g_{3,7}=101 \ ; \ g_{3,8}=100.$$
We choose the binary string $g_{3,6}=111$; in this case, if $m=3$, one has $h+1=6$ and so $h=5$. This means that we are iterating
$$\frac{2^{n+1}}{11}\ \sqrt{\omega(\underbrace{10...0}_{n-3}111)}\, .$$
Hence, for $n=8$:
$$\frac{2^{9}}{11}\ \sqrt{\omega(10000111)}=$$
$$\frac{2^{9}}{11} \ \sqrt{2 - \sqrt{
   2 + \sqrt{
     2 + \sqrt{2 + \sqrt{2 + \sqrt{2 - \sqrt{2 - \sqrt{2 - \sqrt{2}}}}}}}}}\simeq 3.\textbf{14}0996\dots$$
For $n=12$:
$$\frac{2^{13}}{11}\ \sqrt{\omega(100000000111)}=$$
$$\frac{2^{13}}{11} \ \sqrt{2 - \sqrt{
   2 + \sqrt{
     2 + \sqrt{
       2 + \sqrt{
         2 + \sqrt{
           2 + \sqrt{
             2 + \sqrt{
               2 + \sqrt{
                 2 + \sqrt{2 - \sqrt{2 - \sqrt{2 - \sqrt{2}}}}}}}}}}}}}$$
$$\simeq 3.\textbf{14159}0324\dots$$
and so on.
\end{example}

\subsection{The generalized map $M^a_{n}=2a \left(M^a_{n-1}\right)^{2}-\frac{1}{a}$.}

In \cite{Vel} we introduced an extension of the map $L_n$, obtained through the iterated formula $M^a_{n}=2a \left(M^a_{n-1}\right)^{2}-\frac{1}{a} \ , \ a >0$, with $M^a_0(x) = x$. It follows that
\begin{equation}
\displaystyle M^a_0(x) = x \quad ; \quad M^a_1(x) = 2a x^2 - \frac{1}{a} \quad ; \quad M^a_2(x) = 8a^3x^4 - 8ax^2 + \frac{1}{a} \quad \dots .
\end{equation}
Note that the map $L_n$ is a particular case of $M^a_n$, obtained by setting $a=1/2$.
We briefly show that the map $M^a_n$ leads to the same $\pi$ formulas stated in the previous sections.
\begin{proposition}
For $n\geq2$ we have
\begin{equation}
\label{eq:prop1}
M^a_{n}(x)= \frac{1}{a} \cdot \cos(a\ 2^{n}x)+o(x^{2})\, .
\end{equation}
\end{proposition}
\begin{proof}
We must show that:
\begin{equation}
M^a_{n}(x)= \frac{1}{a} - a 2^{2n-1}x^{2}+o(x^{2})
\end{equation}
where we take into account the McLaurin polynomial of cosine. We proceed by induction.
For $n=2$:
\begin{equation}
M^a_{2}(x)= 2a \left(2a x^{2}-\frac{1}{a}\right)^{2}-\frac{1}{a}= \frac{1}{a} - 8ax^{2}+o(x^{2})
\end{equation}
Let us consider the second order McLaurin polynomial of $\frac{1}{a} \cdot \cos(4ax)$: it is just $\frac{1}{a} - 8ax^{2}+o(x^{2})$, thus verifying the relation for $n=2$. Let us now assume (\ref{eq:prop1}) is true for a generic $n$, and deduce that it is also true for $n+1$:
\begin{align}
&M^a_{n+1}=2a \left(M^a_{n}\right)^{2}-\frac{1}{a}=2a\left[\frac{1}{a} - a 2^{2n-1}x^{2}+o(x^{2}) \right]^{2}-\frac{1}{a}=\notag \\
&=\frac{1}{a} - a 2^{2n+1}x^{2}+o(x^{2})
\end{align}
which is in fact the McLaurin polynomial of $\frac{1}{a} \cdot \cos(a\ 2^{n+1}x)$.
\end{proof}
\begin{proposition}
At each iteration the zeros of the map $M^a_n (n \geq 1)$ have the form
\begin{equation}
\label{eq:prop2}
\pm \frac{1}{2a}\cdot \sqrt{2\pm\sqrt{2\pm\sqrt{2\pm\sqrt{2\pm...\pm\sqrt{2}}}}}\, .
\end{equation}
\end{proposition}
\begin{proof}
It is obvious that at $n=1$ this statement is valid.

Now assume that the (\ref{eq:prop2}) is valid for $n$. We have to prove that it is valid for $n+1$.

\begin{equation}
\label{eq:dim_prop2}
2a x^{2}-\frac{1}{a}= \pm \frac{1}{2a}\cdot \sqrt{2\pm\sqrt{2\pm\sqrt{2\pm\sqrt{2\pm...\pm\sqrt{2}}}}}
\end{equation}
or
\begin{equation}
\label{eq:dim_prop2_next}
x^{2}=\frac{1}{2a^{2}} \pm \frac{1}{4a^{2}}\cdot \sqrt{2\pm\sqrt{2\pm\sqrt{2\pm\sqrt{2\pm...\pm\sqrt{2}}}}}
\end{equation}
and placing under the radical sign
\begin{equation}
\label{eq:dim_prop2_next_next}
x=\pm \sqrt{\frac{1}{2a^{2}} \pm \frac{1}{4a^{2}}\cdot \sqrt{2\pm\sqrt{2\pm\sqrt{2\pm\sqrt{2\pm...\pm\sqrt{2}}}}}}
\end{equation}
the thesis is obtained.
\end{proof}
It is possible to prove that zeros of the map $M^a_{n+1}$ are related to those of $M^a_{n}$, $n\geq 1$.

\subsection{$\pi$-formulas: not only approximations.}

From (\ref{theta_value}) and (\ref{Ln_value}) we obtained \cite{Vel} the following formula:
\begin{equation}
\label{eq:rif1}
L_{n}(x)=2\cos\left[2^{n-1}\arctan\left(\frac{\sqrt{1-\left(\frac{x^{2}}{2}-1\right)^2}}{\frac{x^{2}}{2}-1} \right)\right]\, ,
\end{equation}
valid for $x\in[-2,2]$ and $x\neq \pm \sqrt{2}$. This expression is equivalent to
\begin{equation}
\label{eq:rif2}
L_{n}(x)=\left(\left(\left(x^{2}-2\right)^{2}-2\right)^{2}\ ...\ -2\right)^{2}-2
\end{equation}

Moreover, we already observed that, for $|x| = \sqrt{2}$, we have
\begin{equation}
L_0(\sqrt{2}) = \sqrt{2} \quad ; \quad L_1(\sqrt{2}) = 0 \quad ; \quad L_2(\sqrt{2}) = -2 \quad ; \quad L_n(\sqrt{2}) = 2 \ \ \forall n \geq 3  \ .
\end{equation}

The right hand side of (\ref{eq:rif1}) vanishes when
\begin{equation}
\label{eq:atan}
2^{n-1}\arctan\left[\frac{\sqrt{1-\left(\frac{x^{2}}{2}-1\right)^2}}{\frac{x^{2}}{2}-1} \right]=\pm\frac{\pi}{2}\ (2h+1) \ ; \ h \in N \ ; \ x \neq \pm \sqrt{2}
\end{equation}
i.e.,
\begin{equation}
- \frac{\pi}{2} < \arctan\left[\frac{\sqrt{1-\left(\frac{x^{2}}{2}-1\right)^2}}{\frac{x^{2}}{2}-1} \right]=\pm\frac{\pi}{2^{n}}\ (2h+1) < \frac{\pi}{2} \quad , \quad x \neq \pm \sqrt{2}
\end{equation}
whence
\begin{equation}
\label{hmax}
\sqrt{1-\left(\frac{x^{2}}{2}-1\right)^2}= \left(\frac{x^{2}}{2}-1\right) T_{n,h}
\end{equation}
where $T_{n,h}=\tan\left[\pm\frac{\pi}{2^{n}}\ (2h+1) \right]$, for $h=0,1,2..., h_{max}$, and $h_{max}$ defined in this way: from (\ref{eq:atan}) and boundedness of  inverse tangent function we have
$$\frac{\pi}{2^n}(2h+1) < \frac{\pi}{2}$$
from which
$$h < 2^{n-2}-\frac{1}{2}$$
therefore $h_{max}=2^{n-2}-1$, for $n \geq 2$.

If the factor $T_{n,h}$ is negative, the solutions  of (\ref{hmax}) belong to the interval $ (-\sqrt{2},\sqrt{2})$; otherwise $x\in [-2,-\sqrt{2}) \ \cup \ (\sqrt{2},2]$, if $T_{n,h}>0$. We have:
\begin{equation}
1-\left(\frac{x^{2}}{2}-1\right)^2= \ \left(\frac{x^{2}}{2}-1\right)^{2} T_{n,h}^{2} \ \Rightarrow \ \frac{x^{2}}{2}-1=\pm \frac{1}{\sqrt{1+T_{n,h}^{2}}}
\end{equation}
Therefore we can write the zeros of $L_n$ in the form
\begin{equation}
\label{eq:uu}
x^n_h=\pm\sqrt{2\pm \frac{2}{\sqrt{1+\tan^{2}\left[\frac{\pi}{2^{n}}\ (2h+1) \right]}}} \quad , \quad n \geq 2\ ; \ 0\leq h\leq 2^{n-2}-1
\end{equation}
Moreover, we know that, for every $n \geq 2$, the $h$-th positive zero of $L_{n}(x)$ has the form:
\begin{equation}
\sqrt{\omega(g_{n-1,2^{n-1}-h})}
\end{equation}
where $0\leq h\leq 2^{n-2}-1$. Equating the two expressions, one finds:

\begin{equation}
\frac{1}{1+\tan^{2}\left[\frac{\pi}{2^{n}}\ (2h+1) \right]}=\left[\frac{1}{2}\omega(g_{n-1,2^{n-1}-h})-1\right]^{2}
\end{equation}
whence
\begin{equation}
\label{eq:exact}
\pi=\frac{2^{n}}{2h+1}\arctan\sqrt{\frac{1}{\left[\frac{1}{2}\omega(g_{n-1,2^{n-1}-h})-1\right]^{2}}-1} \, , 
\end{equation}
for $n \geq 2$, $0\leq h\leq 2^{n-2}-1$. In this way we obtain infinitely many formulas giving $\pi$ not as the limit of a sequence, but through an equality involving the zeros of the polynomials $L_n$.

Similar considerations can be made for the polynomials $M^a_n$.
Since, for $\displaystyle |x| \neq \frac{\sqrt{2}}{2a}$,

\begin{equation}
M^a_{n}(x)=\frac{1}{a}\cos\left(2^{n-1}\arctan\left[\frac{\sqrt{1-\left(2a^{2}x^{2}-1\right)^2}}{2a^{2}x^{2}-1} \right]\right)
\end{equation}
vanishes if
\begin{equation}
2^{n-1}\arctan\left[\frac{\sqrt{1-\left(2a^{2}x^{2}-1\right)^2}}{2a^{2}x^{2}-1} \right]=\pm\frac{\pi}{2}\ (2h+1)
\end{equation}
i.e.,
\begin{equation}
\arctan\left[\frac{\sqrt{1-\left(2a^{2}x^{2}-1\right)^2}}{2a^{2}x^{2}-1} \right]=\pm\frac{\pi}{2^{n}}\ (2h+1) \ ,
\end{equation}
then
\begin{equation}
\label{diseguaglianza_con_Tnh}
\sqrt{1-\left(2a^{2}x^{2}-1\right)^2}= \left(2a^{2}x^{2}-1\right) T_{n,h}
\end{equation}
where $T_{n,h}=\tan\left[\pm\frac{\pi}{2^{n}}\ (2h+1) \right]$, with $h=0,1,2..., 2^{n-2}-1$.

Furthermore:
\begin{equation}
\left(2a^{2}x^{2}-1\right) \ T_{n,h} >0
\end{equation}
The inequality $2a^{2}x^{2}-1>0$ is verified for $x<-\frac{\sqrt{2}}{2a}\ \vee \ x>\frac{\sqrt{2}}{2a}$. If $T_{n,h}$ is negative, the solutions of (\ref{diseguaglianza_con_Tnh}) belong to the interval $ \left(-\frac{\sqrt{2}}{2a},\frac{\sqrt{2}}{2a}\right)$, otherwise $x\in \left[-\frac{1}{a},-\frac{\sqrt{2}}{2a}\right) \ \cup \ \left(\frac{\sqrt{2}}{2a},\frac{1}{a}\right]$, if $T_{n,h}>0$. On the other hand:
\begin{equation}
1-\left(2a^{2}x^{2}-1\right)^2=T_{n,h}^{2}\ \left(2a^{2}x^{2}-1\right)^{2}\ \Rightarrow \ 2a^{2}x^{2}-1=\pm \frac{1}{\sqrt{1+T_{n,h}^{2}}}
\end{equation}
from which:
\begin{equation}
x_h^n=\pm\frac{1}{2a}\sqrt{2\pm \frac{2}{\sqrt{1+\tan^{2}\left[\frac{\pi}{2^{n}}\ (2h+1) \right]}}}
\end{equation}
Since, from (\ref{eq:prop2}), the zeros of $M^a_{n}(x)$ are proportional to the zeros of $L_{n}(x)$, we can say that also \emph{the $2^{n-1}$ positive zeros of $M^a_{n}$, in decreasing order, follow the order given by the Gray code}:
\begin{equation}
\frac{1}{2a}\sqrt{\omega(g_{n-1,2^{n-1}-h})}
\end{equation}
Equating the two expressions we find again the identity:
\begin{equation}
\pi=\frac{2^{n}}{2h+1}\arctan\sqrt{\frac{1}{\left[\frac{1}{2}\omega(g_{n-1,2^{n-1}-h})-1\right]^{2}}-1}
\end{equation}

\section{Discussion and perspectives.}
\label{sec:conc}

In previous papers (\cite{Vel} and \cite{Vel2}) we introduced a class of polynomials which follow the same recursive formula as the Lucas-Lehmer numbers, studying the distribution of their zeros and remarking that this distributions follows a sequence related to the binary Gray code. It allowed us to give an order for all the zeros of every polynomial $L_n$, \cite{Vel2}. In this paper, the zeros, expressed in terms of nested radicals, are used to obtain two formulas for $\pi$: the first (i.e., formula (\ref{eq:theorem_gray})) can be seen as a generalization of the known formula (\ref{eq:Munkhammar}), because the latter can be seen as the case related to the smallest positive zero of $L_n$; the second (i.e., formula (\ref{eq:exact})) gives infinitely many formulas reproducing $\pi$ not as the limit of a sequence, but through an equality involving the zeros of the polynomials $L_n$.

The proof of the $\pi$-formulas is based on Proposition \ref{prop:terza}. Actually, Proposition \ref{prop:terza} can be fundamental for further studies, too. In fact, it not only allows to get the main results of this paper, but also allows the evaluation of nested square roots of 2 as:
\begin{equation*}
\sqrt{\omega (\bar b_{n-m} g_{m,h+1})}    =\sqrt{2-\sqrt{2+\sqrt{2+\dots+\sqrt{2\pm\sqrt{2\pm\dots\pm\sqrt{2}}}}}}
\end{equation*}
where $\bar b_{n-m}=10\dots0$ has $n-m-1$ zeros, for every $h\in\mathbb N$ such that $h\in\{0,1,\dots,2^m-1\}$ and $n>m+1$. This is a result to put in evidence and to generalize in future researches, for example following interesting insights suggested by paper \cite{Zim08}, where the authors defined the set $S_2$ of all continued radicals of the form
$$a_0\sqrt{2+a_1\sqrt{2+a_2\sqrt{2+a_3\sqrt{2+\dots}}}}$$
(with $a_0 = 1$, $a_k\in\{-1,1\}$ for $k=0,1,\dots,n-1$) and investigated some of its properties by assuming that the limit of the sequence of radicals exists.


\bibliographystyle{te}
\bibliography{pi}

\end{document}